\documentclass[a4paper,10pt]{amsart}
\usepackage{amsmath,amsthm,amsfonts,amssymb,calrsfs}

\usepackage[all]{xy}
\usepackage{tikz-cd, tikz}

\usetikzlibrary{arrows.meta,calc}

\usepackage{changepage}
\usepackage{float}
\usepackage{graphicx}
\usepackage{yfonts}
\usepackage{xcolor}

\usepackage{array}
\usepackage{multirow}
\usepackage{caption}

\usepackage{todonotes}

\usepackage[hyperindex=true, colorlinks=true, linkcolor=blue, filecolor=black,
citecolor=blue, urlcolor=blue,
pagebackref=true,
bookmarks=true, bookmarksopen=true, bookmarksnumbered=true,
bookmarksopenlevel=2, pdfborder={0 0 0}]{hyperref}
\usepackage{cleveref}

\theoremstyle{plain}{
    \newtheorem{theorem}{Theorem}[section]
    \newtheorem{lemma}[theorem]{Lemma}
    
    \newtheorem{proposition}[theorem]{Proposition}
    
}

\theoremstyle{definition}{
    
    \newtheorem{defn-thm}[theorem]{Definition-Theorem}

}

\theoremstyle{remark}{
    \newtheorem{remark}[theorem]{Remark}

}

\def\rank{\mathrm{rank}}

\def\cA{\mathcal{A}}

\def\cF{\mathcal{F}}
\def\cE{\mathcal{E}}
\def\cO{\mathcal{O}}

\def\cT{\mathcal{T}}

\def\cB{\mathcal{B}}

\def\cQ{\mathcal{Q}}

\def\k{\mathfrak{k}}

\def\R{\mathbb{R}}
\def\Q{\mathbb{Q}}
\def\Z{\mathbb{Z}}

\def\P{\mathbb{P}}

\def\C{\mathbb{C}}

\def\>{\rangle}

\def\<{\langle}
\def\>{\rangle}

\def\Vol{\mathrm{Vol}}
\def\mult{\mathrm{mult}}

\DeclareFontFamily{U}{mathc}{}
\DeclareFontShape{U}{mathc}{m}{it}%
{<->s*[1.03] mathc10}{}
\DeclareMathAlphabet{\mathcal}{U}{mathc}{m}{it}

\begin{document}

\title[Subspace concentration and BG inequality]{Subspace concentration condition
and Bogomolov--Gieseker inequality}

\address{Univ Brest, UMR CNRS 6205, Laboratoire de Mathématiques de Bretagne
Atlantique, France}

\author[C. Tipler]{Carl TIPLER}
\email{carl.tipler@univ-brest.fr}


\begin{abstract}
We provide a necessary condition for  integral polytopes to satisfy subspace concentration type inequalities, by mean of a convex geometric interpretation of the Bogomolov--Gieseker inequality.
\end{abstract}

\maketitle

\section{Introduction}
A finite Borel measure $\mu$ on the sphere $S^{n-1}$ is said to satisfy the \textit{subspace concentration inequalities} (SCI) if for any subspace $E\subset \R^n$,
\begin{equation}
 \label{eq:SCC for sphere measures}
 \mu(E\cap S^{n-1})\leq \frac{\dim(E)}{n} \, \mu(S^{n-1}).
\end{equation}
If in addition equality holds for $E\subset \R^n
$ if and only if there exists a decomposition $E\oplus G = \R^n$  such that $n\,\mu(G\cap S^{n-1})= \dim(G) \, \mu(S^{n-1})$, $\mu$ satisfies the \textit{subspace concentration condition} (SCC). Those conditions play an important role in convex geometry. The SCC characterizes cone-volume measures of origin-symmetric convex bodies amongst finite even Borel measures on the unit sphere \cite{BoLuYaZha}, and is satisfied by cone-volume measures of polytopes with centroid at the origin \cite{HenkLinke}.

In this note, we introduce a subspace concentration condition for integral polytopes, and show that it imposes constraints on a weighted average measure of the codimension $2$ facets, where the weights are determined by angles between facet normals. Let $N$ be a rank $n$ lattice with dual $M=\mathrm{Hom}_\Z(N,\Z)$ and duality pairing $\langle\cdot,\cdot\rangle$. A lattice polytope $P\subset M_\R$ is the convex hull in $M_\R:=M\otimes_\Z\R$ of a finite number of points in $M$ (we will always assume $P$ to be full-dimensional). Equivalently, $P$ is a finite intersection of half-spaces
\begin{equation}
 \label{eq:facet presentation}
P=\lbrace m\in M_\R \:\vert\: \forall\, \rho\in\Sigma(1),\, \langle m , u_\rho \rangle \geq -a_\rho \rbrace,
\end{equation}
for some integers $(a_\rho)_{\rho\in\Sigma(1)}$ and primitive inward pointing facet normals $(u_\rho)_{\rho\in\Sigma(1)}$ in $N$ (at that stage $\Sigma(1)$ is just a finite set, but we use standard toric geometry notations for later purposes).
 We will denote by $F^\rho$ the facet of $P$ with inward normal $u_\rho$, and we index by $\Sigma(2)$ the codimension $2$ faces of $P$, that we may assume of the form $F^\tau=F^\rho\cap F^{\rho'}$ for suitable $(\rho,\rho')\in\Sigma(1)^2$. For such a face $F^\tau$, we introduce the multiplicity $m_\tau=\mult(\tau)$  to be the index of the lattice spaned by $(u_\rho,u_{\rho'})$ in the lattice $(\R\cdot u_\rho +\R\cdot u_{\rho'})\cap N$. We also introduce an invariant $\cA_\tau\in\Q$ as follows. Subdivide the two dimensional cone $\tau:=\R_+\cdot u_\rho +\R_+\cdot u_{\rho'}\subset N_\R$ into $m_\tau$ cones by considering the rays in $\tau$ spaned by elements in
\[
N\cap\lbrace  \lambda u_\rho +\lambda' u_{\rho'}\:\vert\: 0\leq \lambda<1,\: 0\leq \lambda' <1 \rbrace,
\]
see Figure \ref{fig:subdivision} in Section \ref{sec:figures}. By choosing an integral basis $\cB$ for $N$, we can fix an integral scalar product\footnote{$Q$ is said integral when it takes integral values on $N$} $Q$ on $N_\R=N\otimes_\Z\R$ making $\cB$ orthonormal and inducing an isomorphism $M_\R\simeq N_\R$.
Then, we set
\begin{equation}
 \label{eq:Alpha tau}
\cA_\tau :=\sum_{\substack{ \tau'\subset \tau\\\tau'=\rho_1+\rho_2}} \Big( 2+(n-1)\Big(\frac{Q(u_{\rho_1},u_{\rho_2})}{Q(u_{\rho_1},u_{\rho_1})}+\frac{Q(u_{\rho_2},u_{\rho_1})}{Q(u_{\rho_2},u_{\rho_2})}\Big)\Big)
\end{equation}
where the sum is over the $m_\tau$ cones subdividing $\tau$, and $u_{\rho_i}$ is a primitive generator of the ray $\rho_i$. Given a face $F\subset P$, we will denote by $\mathrm{Vol}_M(F)$
its lattice volume in the affine span of $F$ with respect to the lattice $\mathrm{span}(F)\cap M$.
\begin{theorem}
 \label{theointro:main theorem}
 Let $P\subset M_\R$ be a lattice polytope with facet presentation \eqref{eq:facet presentation}. Assume that the following subspace concentration inequalities hold: for any subspace $E\subset N_\R$,
 \begin{equation}
  \label{eq:SCC for lattice polytopes}
   \sum_{u_\rho\in E} \: \Vol_M(F^\rho) \leq \frac{\dim(E)}{n} \: \sum_{\rho\in\Sigma(1)} \: \Vol_M(F^\rho).
 \end{equation}
Then,
\begin{equation}
 \label{eq:BG for lattice polytopes}
 \sum_{\tau\in\Sigma(2)}\cA_\tau \: \Vol_M(F^\tau) \geq 0.
\end{equation}
\end{theorem}
The proof of Theorem \ref{theointro:main theorem} is a translation into convex geometric terms that a slope semistable tangent sheaf satisfies the Bogomolov--Gieseker inequality, by mean of toric geometry.  Note that  we do not assume $P$ to be smooth, so the argument requires several minor generalisations of classical results.

 The subspace concentration condition considered in Theorem \ref{theointro:main theorem} differs slightly from the classical one from convex geometry, as we now explain. Given a polytope $P\subset M_\R\simeq \R^n$ with the origin in its interior, the cone-volume measure of $P$ is defined as follows (see Figure \ref{fig:polytopes} in Section \ref{sec:figures}). For any facet $F\subset P$ with outter unit normal $v_F\in S^{n-1}$, we let $\Vol(C_F)$ to be the euclidean volume of the cone $C_F$ with basis $F$ and apex $0$. Then, the cone-volume measure $\mu_P$ of $P$ is a discrete measure on $S^{n-1}$ defined by
 \[
  \mu_P=\sum_{F\subset P} \, \Vol(C_F)\, \delta_{v_F},
 \]
 where the sum runs over facets of $P$. It satisfies the SCI \eqref{eq:SCC for sphere measures} if for any subspace $E\subset \R^n$, one has
 \begin{equation}
  \label{eq:SCI for cone vol measure}
\sum_{v_F\in E} \: \Vol(C_F) \leq \frac{\dim(E)}{n}\, \sum_{F\subset P} \, \Vol(C_F).
 \end{equation}
 Assume now that $P\subset M_\R$ is a lattice polytope with facet presentation \eqref{eq:facet presentation}. We consider the euclidean structure on $\R^n\simeq N_\R\simeq M_\R$ induced by $Q$. By definitions and the pyramid formula, for any facet $F^\rho$, one has
 \begin{equation*}
\Vol(C_{F^\rho})  = \displaystyle\frac{1}{n}\,\vert a_\rho \vert\,\Vol_M(F^\rho).
 \end{equation*}
Then, the SCI \eqref{eq:SCC for lattice polytopes} corresponds to the SCI \eqref{eq:SCI for cone vol measure} when all the $a_\rho$'s are equal to $1$. In this case $P$ is called a \textit{reflexive} polytope. Relying on deep results from gauge theory and algebraic geometry, namely stability of the tangent sheaves of  $K$-polystable smooth Fano toric varieties, it was observed in \cite{HNS19} that smooth  reflexive polytopes with barycenter at the origin satisfy the SCI, which is a special case of the main result in \cite{HenkLinke}.  For smooth reflexive polytopes, Theorem \ref{theointro:main theorem} can be rephrased as follows.
\begin{theorem}
 \label{theointro:Fano case}
 Let $P\subset M_\R$ be a smooth reflexive lattice polytope with facets $(F^\rho)_{\rho\in\Sigma(1)}$ and primitive inward pointing facet normals $(u_\rho)_{\rho\in\Sigma(1)}$. Assume that the following subspace concentration inequalities hold: for any subspace $E\subset N_\R$,
 \begin{equation}
  \label{eq:SCC for reflexive polytopes}
   \sum_{u_\rho\in E} \: \Vol_M(F^\rho) \leq \dim(E) \:  \Vol_M(P).
 \end{equation}
Then
\begin{equation}
 \label{eq:BG for reflexive polytopes}
 \sum_{\tau\in\Sigma(2)}\, \Vol_M(F^\tau) \geq \frac{(n-1)^2}{2}\,\Vol_M(P).
\end{equation}
 \end{theorem}
The conclusion of Theorem \ref{theointro:Fano case}, Inequality \eqref{eq:BG for reflexive polytopes}, is rather surprising, at least to the author, from the point of view of volume polynomials. Recall that the \textit{volume polynomial} $f_P$ of $P$ is defined by
\[
f_P  : (x_\rho)_{\rho\in\Sigma(1)} \mapsto \Vol_M(P(x_\rho)),
\]
where
\[
 P(x_\rho)=\lbrace m\in M_\R \:\vert\: \langle m , u_\rho \rangle \geq -x_\rho \rbrace.
\]
It is known that the volume polynomial is log concave. As the facet volumes $\Vol_M(F^\rho)$ appear to be first order derivatives of $f_P$, the inequalities \eqref{eq:SCC for reflexive polytopes} are upper bounds for directional derivatives of $\log(f_P)$. On the other hand, the volumes of codimension $2$ faces are second order mixed derivatives, and \eqref{eq:BG for reflexive polytopes} provides a \textit{lower bound} for the off-diagonal terms of the Hessian of $\log(f_P)$. While our proof relies on intersection theoretical interpretations of those inequalities, it would be interesting to obtain a direct proof using convex geometry. Indeed, many features of volume polynomials in algebraic geometry are shared with volume polynomials of convex bodies \cite{Huh26}. The results in this note then raise a natural question : is there a Bogomolov--Gieseker type inequality for convex bodies? While a scaling argument shows that Theorem \ref{theointro:main theorem} can be generalised to rational polytopes, an extension to general convex bodies will require a different approach.

\subsection*{Acknowledgments}
The author would like to thank Christophe Mourougane for raising the question of the Bogomolov--Gieseker inequality for toric sheaves, which ultimately lead to this note. The author is partially supported by the grants MARGE ANR-21-CE40-0011 and BRIDGES ANR-FAPESP ANR-21-CE40-0017. He also beneficiated from the support of the French government “Investissements d'Avenir” program integrated to France 2030, under the reference ANR-11-LABX-0020-01.

 \subsection*{AI declaration use} This article was written during summer 2026, before the creation of the Association for Human Mathematics (AHM). A LLM was used to help producing the figures in the last section. Apart from these figures, all the content of this paper was produced by the author, who fully supports the AHM. 
 
\section{Preliminaries}
\label{sec:background}
We will use standard notations from toric geometry, and refer to \cite{Ful93,CLS} for the basics. In particular, $\Sigma$ will stand for a fan of strongly convex rational polyhedral cones in $N_\R$, and $\Sigma(k)$ for its $k$-dimensional cones. The fan $\Sigma$ determines a normal toric variety $X$, and to each $\sigma\in\Sigma$ we associate the orbit closure $V(\sigma)\subset X$ \cite[Section 3.2]{CLS}. In particular, we will denote $D_\rho=V(\rho)$ the torus invariant divisor associated to $\rho\in\Sigma(1)$.

\subsection{Lattice polytopes, toric varieties, and intersections}
To a lattice polytope $P\subset M_\R$ with facet presentation \eqref{eq:facet presentation}, one may associate a normal projective toric variety $X$ together with an ample Cartier divisor
\[
D=\sum_{\rho\in\Sigma(1)} a_\rho\, D_\rho \subset X
\]
in such a way that the normal fan $\Sigma$ of $P$ is the fan $\Sigma_X$ of $X$ \cite[Chapter 5]{Ful93}. We will use a toric and projective resolution of singularities
\[
 f: X'\to X
\]
where $X'$ is a smooth and complete toric variety with fan $\Sigma'$ obtained from $\Sigma$ by successive star subdivisions, as in \cite[Theorem 11.1.9, Proposition 11.1.7]{CLS}. Note in particular that $f$ is proper. We can pullback the Cartier divisor
\[
D'=f^*D =\sum_{\rho\in\Sigma'(1)} b_{\rho}\, D_{\rho} \subset X'
\]
to $X'$. This divisor $D'$ is a big and nef divisor on $X'$ (nefness of $D'$ comes from global generation \cite[Theorem 6.3.12]{CLS} and bigness from the fact that $P$ is full dimensional). In this setting, we have the following proposition for intersection numbers on $X'$.
\begin{proposition}
 \label{prop:intersection numbers}
 Let $\rho\in\Sigma'(1)$ and $\tau'\in\Sigma'(2)$. Then the following holds.
 \begin{enumerate}
  \item Either $\rho\in\Sigma(1)$ and then
  \[
   D_\rho\cdot (D')^{n-1}=(n-1)!\,\Vol_M(F^\rho),
  \]
or $\rho\notin\Sigma(1)$ and then
\[
D_\rho\cdot (D')^{n-1}=0.
\]
\item Either there exists $\tau\in\Sigma(2)$ with $\tau'\subset \tau$ and then
\[
   [V(\tau')]\cdot (D')^{n-2}=(n-2)!\,\Vol_M(F^\tau),
  \]
  or
  \[
[V(\tau')]\cdot (D')^{n-2}=0.
   \]
 \end{enumerate}
\end{proposition}
\begin{proof}
 The results are certainly well known, but we recall the proof for the sake of completeness. By the push--pull formula, we have
 \[
   D_\rho\cdot (D')^{n-1}=f_*D_\rho\cdot D^{n-1}
 \]
and
  \[
    [V(\tau')]\cdot (D')^{n-2}=f_* [V(\tau')]\cdot D^{n-2}.
 \]
 Using the properties of pushforward of cycles (see the discussion at the end of \cite[Section 5.1]{Ful93}), we find that either $\rho\in\Sigma(1)$ and then $f_*D_\rho=D_\rho\subset X$, or $f_* D_\rho=0$. Similarly for $\tau'$, either there exists $\tau\in\Sigma(2)$ with $\tau'\subset \tau$ and then $f_*[V(\tau')]=[V(\tau)]$ or $f_*[V(\tau')]=0$. To conclude, we use the fact that from \cite[Chapter 5]{Ful93}, the intersection theory on $X$ yields
\begin{equation}
 \label{eq:intersection D rho}
\Vol_M(F^\rho)=\frac{1}{(n-1)!}\, D_\rho\cdot D^{n-1}
\end{equation}
and for $\tau\in\Sigma(2)$,
\begin{equation}
 \label{eq:intersection D tau}
\Vol_M(F^\tau)= \frac{1}{(n-2)!}\, [V(\tau)]\cdot D^{n-2}.
\end{equation}
Note that for Equations \eqref{eq:intersection D rho} and \eqref{eq:intersection D tau}, the arguments from \cite[Section 5.3]{Ful93} extend to the normal setting.
\end{proof}

The equalities from Proposition \ref{prop:intersection numbers} will be used to translate the inequalities \eqref{eq:SCC for lattice polytopes} and \eqref{eq:BG for lattice polytopes} into intersection theoretical inequalities on $X'$. We will also use the following formula for the total Chern class of $X'$ \cite[Proposition 13.1.2]{CLS}:
\begin{equation}
 \label{eq:total chern class}
 c(X')=\sum_{\sigma\in\Sigma'}\, [V(\sigma)].
\end{equation}
In particular,
\[
 c_1(X')=\sum_{\rho\in\Sigma'(1)}\, [D_\rho]
\]
and 
\[
 c_2(X')=\sum_{\tau'\in\Sigma'(2)}\, [V(\tau')].
\]

\subsection{SCI and stability of the tangeant sheaf}
It is well known from specialists that the SCI \eqref{eq:SCC for lattice polytopes} correspond precisely to the slope semistability of the tangent sheaf $\cT_X$ of $X$, at least when $X$ is smooth \cite{HNS19,DDK20}. We explain now how to adapt this result to the case when $X$ is singular, working instead on the smooth model $(X',D')$, with $D'$ big and nef. A torsion--free sheaf $\cE$ on $X'$ is said to be \textit{slope semistable} with respect to
\[
L:=D'
\]
if for any  coherent subsheaf $\cF\subset \cE$ with $0<\rank(\cF)<\rank(\cE)$, one has the slope inequality
\begin{equation}
 \label{eq:slope inequalities}
 \frac{c_1(\cF)\cdot L^{n-1}}{\rank (\cF)}\leq \frac{c_1(\cE)\cdot L^{n-1}}{\rank (\cE)}.
\end{equation}
When strict inequalities always hold, $\cE$ is said to be \textit{slope stable}. The quantity
\[
\frac{c_1(\cF)\cdot L^{n-1}}{\rank (\cF)}
\]
is called the \textit{slope} of $\cF$, and is denoted $\mu_L(\cF)$.
\begin{proposition}
 \label{prop:equiv semistable}
 The tangent sheaf $\cT_{X'}$ is slope semistable with respect to $L$ if and only if the SCI \eqref{eq:SCC for lattice polytopes} holds for any subspace $E\subset N_\R$, where $P\subset M_\R$ is the lattice polytope associated to $(X,D)$.
\end{proposition}

\begin{proof}
By Proposition \ref{prop:intersection numbers}, equation \eqref{eq:intersection D rho} and equation \eqref{eq:total chern class},
 \begin{equation*}
  \begin{array}{ccc}
  \displaystyle  \frac{1}{n}\sum_{\rho\in\Sigma(1)}\,\Vol_M(F^\rho) &  = &\displaystyle \frac{1}{n!}\sum_{\rho\in\Sigma(1)}\,  D_\rho\cdot D^{n-1}\\
  &  = &\displaystyle \frac{1}{n!}\sum_{\rho\in\Sigma'(1)}\,  D_\rho\cdot D^{n-1}\\
     & = &\displaystyle  \frac{1}{n!}\,c_1(\cT_{X'})\cdot D^{n-1}\\
  & = & \displaystyle\frac{\mu_L(\cT_{X'})}{(n-1)!}.
  \end{array}
 \end{equation*}
On the over hand, by Klyachko's work \cite{Kly90}, as the tangeant bundle $\cT_{X'}$ is torus equivariant, it corresponds to a family of filtrations $(T^\rho(j))_{j\in\Z, \rho\in\Sigma'(1)}$ for the vector space $N_\C=N\otimes_\Z\C$. From Perling \cite{Per04}, any saturated coherent equivariant subsheaf $\cE\subset\cT_{X'}$ corresponds to a subfamily of filtrations of the form $(E\cap T^\rho(j))_{j\in\Z, \rho\in\Sigma(1)}$, for some vector subspace $E\subset N_\C$ (see e.g. \cite[Lemma 2.15]{NaTip} for the saturation property). Then, from \cite[Section 2]{HNS19}, the slope of $\cE$ with $0<\rank(\cE)$ is precisely given by
\[
 \mu_L(\cE)=\frac{1}{\dim(E)}\sum_{\substack{\rho\in\Sigma'(1)\\ u_\rho\in E} } \: D_\rho\cdot L^{n-1}.
\]
By Proposition \ref{prop:intersection numbers} again, this equals
\[
 \mu_L(\cE)=\frac{(n-1)!}{\dim(E)}\sum_{ u_\rho\in E } \:\Vol_M(F^\rho).
\]
Hence, the SCI \eqref{eq:SCC for lattice polytopes} are satisfied precisely when $\cT_{X'}$ is slope semistable with respect to \textit{equivariant} subsheaves. This in turn implies semistability \cite{Koo11}. Indeed, as $D$ is nef, there exists by \cite[Section 2.6]{GrebKebPet} a unique Jordan--H\"older filtration for $\cT_{X'}$, that is a filtration by coherent subsheaves
\[
 0=\cE_0\subsetneq \cE_1\subsetneq \cE_1\subsetneq\ldots\subsetneq \cE_d=\cT_{X'}
\]
where each quotient $\cQ_i=\cE_{i}/\cE_{i-1}$ is torsion--free, slope semistable, and where the sequence of slopes $\mu_D(\cQ_i)$ is strictly decreasing. From uniqueness, we deduce that it is equivariant under the torus action. Hence,  semistability, which is characterized by having a single term $\cE_1=\cT_{X'}$ in the Jordan--H\"older filtration, is implied by its equivariant version. This concludes the proof.
\end{proof}
\begin{remark}
 We  work on a smooth model $(X',D')$ for $(X,D)$ to be able to use Chern classes and their intersections. One could use instead Chern numbers on singular varieties as defined in \cite{ClauGrafGuen}, but the approach is essentially equivalent, as those intersection numbers require a resolution to be defined.
\end{remark}

\subsection{Bogomolov--Gieseker inequality for big and nef classes}

A classical result in algebraic geometry is the Bogomolov--Gieseker inequality, that we now recall.
\begin{theorem}
 \label{theo:BG}
 Let $\cE$ be a torsion-free slope semistable sheaf on $(X',D')$. Then
 \begin{equation}
  \label{eq:BG inequality}
  \Delta(\cE)\cdot (D')^{n-2}\geq 0,
 \end{equation}
where $\Delta(\cE)=2\,\rank(\cE)\,c_2(\cE)-(\rank(E)-1)\,c_1^2(\cE)$ is the discriminant of $\cE$.
\end{theorem}
Note that $X'$ is smooth and projective, hence the Chern classes or $\cE$ are well defined. This inequality is a classical result when $D'$ is an ample divisor \cite[Corollary 4.7]{Miyaoka}, and is certainly known to experts when $D'$ is merely big and nef. For convenience of the reader, we recall briefly how to obtain the result in our setting, following \cite{ClauGrafGuen}. First, if $\cE$ is slope stable, one can consider ample $\Q$-divisors $D_\varepsilon=D'+\varepsilon A$, with $\varepsilon\in\Q_+$ small and $A$ ample. By openess of stability with respect to the change of polarisation, $\cE$ is stable with respect to $D_\varepsilon$ for all $\varepsilon$ small enough. Then, it satisfies the Bogomolov--gieseker inequality with respect to $D_\varepsilon$, and the result follows letting $\varepsilon \to 0$. If $\cE$ is only semistable, then it admits a Jordan--H\"older filtration
\[
0=\cE_0\subsetneq \cE_1\subsetneq\ldots\subsetneq \cE_d=\cE
\]
by coherent subsheaves such that the successive quotients $\cQ_i=\cE_{i+1}/\cE_i$ are torsion--free and stable with the same slope as $\cE$, cf \cite[Proposition 3.5]{ClauGrafGuen}. Each $\cQ_i$  satisfies the Bogomolov--Gieseker inequality $\Delta(\cQ_i)\cdot (D')^{n-2}\geq 0$ by the previous case. Then, one uses the fact that for a short exact sequence of torsion-free sheaves
\[
 0\to \cE_1\to \cE_2\to\cE_3\to 0,
\]
with $\mu_D(\cE_1)=\mu_D(\cE_3)$ and
\[
\Delta(\cE_i)\cdot (D')^{n-2}\geq 0
\]
for $i\in\lbrace 1, 3\rbrace$, one has 
\[
 \Delta(\cE_2)\cdot (D')^{n-2}\geq 0,
\] 
which follows from a direct computation and the  Hovanskii--Teissier inequality, cf the proof  of \cite[Lemma 3.7]{ClauGrafGuen}. By induction on the number of quotients in the Jordan--H\"older filtration for $\cE$, the result follows.

 \section{Proof of Theorem \ref{theointro:main theorem}}
 \label{sec:proof main theorem}
 Assume that we are in the setting of Theorem \ref{theointro:main theorem}, and consider the smooth toric variety $X'$ with nef and big class $D'$ as in the previous sections. Then, by Proposition \ref{prop:equiv semistable},  $\cT_{X'}'$ is slope semistable with respect to $D''$. From Bogomolov--Gieseker inequality,
 \[
\Delta(\cT_{X'})\cdot (D')^{n-2}\geq 0,
 \]
 that is
 \begin{equation}
  \label{eq:BG intersection X'}
  2nc_2(X')\cdot (D')^{n-2}-(n-1)c_1^2(X')\cdot (D')^{n-2}\geq 0.
 \end{equation}
 We will now translate this latter inequality in convex geometric terms for $P$. To do so, we first need to control the number of two dimensional cones $\tau'\in\Sigma'(2)$ that are sent into cones $\tau\in\Sigma(2)$.
 \begin{lemma}
  \label{lem:2 cones}
  The resolution $f : X'\to X$ may be chosen such that for any $\tau\in\Sigma(2)$,  $\tau$ is exactly the union of $\mult(\tau)$ cones in $\Sigma'(2)$.
 \end{lemma}
 \begin{proof}
 Note that the resolution $f : X' \to X$ is obtained by successive star subdivisions, see the proofs of \cite[Theorem 11.1.9]{CLS} and \cite[Proposition 11.1.7]{CLS}. The first subdivisions, \cite[Proposition 11.1.7]{CLS}, provide a simplicial model for $X$. By construction, see \cite[Proposition 11.1.6]{CLS}, the star subdivisions involved in the simplicialization may add $2$-dimensional cones to $\Sigma$, but will not subdivide the elements in $\Sigma(2)$. Then, the smooth model is obtained by successive subdivisions of cones $\sigma$ with $\mult(\sigma)>1$, cf proof of  \cite[Theorem 11.1.9]{CLS}. This process can be done by decreasing induction on $\dim(\sigma)$ by \cite[Proposition 11.1.8(d)]{CLS}. When $\dim(\sigma) >2$, the star subdivision involved will not affect the elements in $\Sigma(2)$. Hence, only the end of the process, where $\tau\in\Sigma(2)$ has $\mult(\tau)>1$, will modify the cones in $\Sigma(2)$. In that case, by the proof of \cite[Theorem 11.1.9]{CLS}, and \cite[Proposition 11.1.8(b)]{CLS}, we see that $\tau$ is subdivided in $\mult(\tau)$ cones, which concludes the proof of the lemma.
 \end{proof}

 From the previous Lemma \ref{lem:2 cones}, and by  \eqref{eq:total chern class} and Proposition \ref{prop:intersection numbers} we obtain
 \begin{equation}
  \label{eq:intersection c 2 term}
  c_2(X')\cdot (D')^{n-2}=(n-2)!\,\sum_{\tau\in\Sigma(2)} \mult(\tau)\, \Vol_M(F^\tau).
 \end{equation}
 What is left in \eqref{eq:BG intersection X'} is obtaining a geometric interpretation of the intersection number $c_1^2(X')\cdot (D')^{n-2}$. This could be done by mean of Hirzebruch--Riemann-Roch formula, as $\chi(X,\cO(kD))=\mathrm{Ehr}_P(k)$ is the Ehrhart polynomial of $P$, whose coefficients can be computed thanks to the local Euler--MacLaurin formula of Berline--Vergne \cite{BerVer}. Instead, we use a very simple and direct argument, borrowing an idea from \cite{ChanLeung}, to remove the self intersection terms $D_\rho^2\cdot (D')^{n-2}$. In what follows, we will use  \cite[Lemma 12.5.2]{CLS} which implies that
\[
D_\rho\cdot D_{\rho'}=[V(\tau')]                                                                \]
if $\tau'=\rho+\rho'\in\Sigma'(2)$, while
\[
 D_\rho\cdot D_{\rho'}=0
\]
if $\rho\neq \rho'$ and $\rho+\rho'\notin \Sigma'(2)$.  We have
 \begin{equation*}
  \begin{array}{ccc}
   c_1^2(X')\cdot (D')^{n-2} & = &\displaystyle \Big(\sum_{\rho\in\Sigma'(1)} D_\rho\Big)^2\cdot (D')^{n-2} \\
    & = &\displaystyle 2\Big(\sum_{\substack{\tau'\in\Sigma'(2)\\ \tau'=\rho+\rho'}}  D_\rho\cdot D_{\rho'}\cdot (D')^{n-2}\Big)+ \sum_{\rho\in\Sigma'(1)} D_\rho^2\cdot (D')^{n-2}\\
    & = &\displaystyle 2\Big(\sum_{\tau'\in\Sigma'(2)}  [V(\tau')]\cdot (D')^{n-2}\Big)+ \sum_{\rho\in\Sigma'(1)} D_\rho^2\cdot (D')^{n-2}.
  \end{array}
 \end{equation*}
We may fix for each $\rho\in\Sigma'(1)$ a rational element $m_\rho\in M_\Q=M\otimes_\Z\Q$ such that $\langle m_\rho, u_\rho \rangle =-1$ (recall that $u_\rho$ is a primitive generator of $\rho$). Let $\chi^{m_\rho}$ be the character function associated to $m_\rho$. Then, denoting $\sim_\Q$ the $\Q$-rational equivalence, we have
\[
 D_\rho\sim_\Q D_\rho +\mathrm{div}(\chi^{m_\rho})
\]
and thus
\begin{equation*}
  \begin{array}{ccc}
  \displaystyle \sum_{\rho\in\Sigma'(1)} D_\rho^2 & \sim_\Q &  \displaystyle \sum_{\rho\in\Sigma'(1)} D_\rho\cdot \Big(D_\rho +\mathrm{div}(\chi^{m_\rho}) \Big) \\
    & \sim_\Q & \displaystyle  \sum_{\rho\in\Sigma'(1)} D_\rho\cdot \Big(D_\rho + \sum_{\rho'\in\Sigma'(1)} \langle m_\rho, u_{\rho'}\rangle\, D_{\rho'}\Big)\\
    & \displaystyle\sim_\Q & \displaystyle\sum_{\rho\neq\rho'} \langle m_\rho, u_{\rho'}\rangle\, D_\rho\cdot D_{\rho'}
  \end{array}
 \end{equation*}
where we used
\[
\mathrm{div}(\chi^{m_\rho})=\sum_{\rho'}\langle m_\rho, u_{\rho'}\rangle\, D_{\rho'},
\]
cf \cite[Proposition 4.1.2]{CLS}. We obtain
\[
\displaystyle \sum_{\rho\in\Sigma'(1)} D_\rho^2\sim_\Q  \displaystyle\sum_{\tau'=\rho_1+\rho_2\in\Sigma'(2)} \big(\langle m_{\rho_1}, u_{\rho_2}\rangle+\langle m_{\rho_2}, u_{\rho_1}\rangle\big)\, [V(\tau')].
\]
Altogether, using again Proposition \ref{prop:intersection numbers} and Lemma \ref{lem:2 cones}, we deduce that
\begin{equation}
 \label{eq:quadratic c one term}
 c_1^2(X')\cdot (D')^{n-2} = (n-2)! \sum_{\tau\in\Sigma(2)} \big(2\,\mult(\tau)+\alpha_\tau)\: \Vol_M(F^\tau)
\end{equation}
where
\[
 \alpha_\tau:=\sum_{\substack{\tau'=\rho_1+\rho_2\in\Sigma'(2)\\
 \tau'\subset\tau}} \langle m_{\rho_1}, u_{\rho_2}\rangle+\langle m_{\rho_2}, u_{\rho_1}\rangle.
\]
We use now the integral scalar product $Q$ on $N_\R$, which enables to identify $M_\R$ with $N_\R$, and choose $m_\rho\simeq -\frac{u_\rho}{Q(u_\rho,u_\rho)}$. This yields
\begin{equation}
 \label{eq:alpha tau}
 \alpha_\tau=-\sum_{\substack{\tau'=\rho+\rho'\in\Sigma'(2)\\
 \tau'\subset\tau}} \Big(\frac{Q(u_\rho,u_{\rho'})}{Q(u_\rho,u_\rho)}+\frac{Q(u_{\rho'},u_{\rho})}{Q(u_{\rho'},u_{\rho'})}\Big).
\end{equation}
Then, Inequality \eqref{eq:BG intersection X'}, with \eqref{eq:intersection c 2 term} and \eqref{eq:quadratic c one term} implies
\begin{equation}
 \label{eq:final}
 \sum_{\tau\in\Sigma(2)} \big( 2\,\mult(\tau)-(n-1)\alpha_\tau \big) \Vol_M(F^\tau) \geq 0.
\end{equation}
From \eqref{eq:alpha tau},
\[
 \cA_\tau= 2\,\mult(\tau)-(n-1)\alpha_\tau,
\]
and
this concludes the proof of Theorem \ref{theointro:main theorem}.

 \section{Proof of Theorem \ref{theointro:Fano case}}
 We specify the setting to a smooth reflexive polytope $P$. In that case, $X'=X$ is smooth and Fano, with $D'=D=-K_X$.  The degree of $\cT_X$ is given by
 \begin{equation*}
 \begin{array}{ccc}
 c_1(\cT_X)\cdot D^{n-1}  & = & D^{n}\\
  & = & n!\, \Vol_M(P).
  \end{array}
 \end{equation*}
 Hence, the SCI \eqref{eq:SCC for reflexive polytopes} is nothing but equivariant stability of $\cT_X$ with respect to $D$. We deduce the Bogomolov--Gieseker inequality
\[
 \Delta(\cT_X)\cdot D^{n-2}\geq 0.
\]
Using \eqref{eq:total chern class} and Proposition \ref{prop:intersection numbers}, we have
\begin{equation*}
 \begin{array}{ccc}
 \Delta(\cT_X)\cdot D^{n-2} & = & 2nc_2(X)\cdot D^{n-2}-(n-1)c_1^2(X)\cdot D^{n-2}\\
  & = & \displaystyle 2n(n-2)!\,\sum_{\tau\in\Sigma(2)}\,\Vol_M(F^\tau) -(n-1)n!\, \Vol_M(P),
  \end{array}
 \end{equation*}
and the result follows.

\section{Figures}
\label{sec:figures}
%
This section contains the two figures used to illustrate the introduction. 

\subsection{AI use declaration} An AI language model, ChatGPT (OpenAI), was used to help preparing these figures in LaTeX. The author takes full responsibility for this content.

\subsection{Lattice polytopes presentations} We give here the two presentations for lattice polytopes facets normals. The first picture in Figure \ref{fig:polytopes} has outward pointing unit   normals, while the second picture shows the inward pointing primitive generators for the normal rays.

\begin{figure}[ht]
\centering

\begin{tikzpicture}[scale=0.75,xshift=-0.8cm,>=Latex]

\foreach \x in {-4,...,4}
  \foreach \y in {-3,...,3}
    \fill[gray!60] (\x,\y) circle (1.3pt);

\draw[->] (-4.3,0) -- (4.4,0) node[right] {$x$};
\draw[->] (0,-3.3) -- (0,3.4) node[above] {$y$};

\coordinate (A) at (-3,0);
\coordinate (B) at (-2,2);
\coordinate (C) at (1,2);
\coordinate (D) at (3,0);
\coordinate (E) at (2,-2);
\coordinate (F) at (-1,-2);
\coordinate (O) at (0,0);

\fill[orange!12] (A)--(B)--(C)--(D)--(E)--(F)--cycle;
\draw[thick] (A)--(B)--(C)--(D)--(E)--(F)--cycle;

\foreach \P/\Q in {A/B,B/C,C/D,D/E,E/F,F/A}{
  \draw[dashed,gray] (O)--(\P);
  \draw[dashed,gray] (O)--(\Q);
}

\foreach \P in {A,B,C,D,E,F}
  \fill (\P) circle (2.2pt);
\fill (O) circle (2.6pt) node[below right] {$0$};

\draw[->,red,thick]
  ($(B)!0.5!(C)$) -- ++(0,1)
  node[left] {$v_1=(0,1)$};

\draw[->,blue,thick]
  ($(C)!0.5!(D)$) -- ++({1/sqrt(2)},{1/sqrt(2)})
  node[right] {$v_2=\frac1{\sqrt2}(1,1)$};

\draw[->,green!60!black,thick]
  ($(D)!0.5!(E)$) -- ++({2/sqrt(5)},{-1/sqrt(5)})
  node[right] {$v_3=\frac1{\sqrt5}(2,-1)$};

\draw[->,purple,thick]
  ($(E)!0.5!(F)$) -- ++(0,-1)
  node[below] {$v_4=(0,-1)$};

\draw[->,orange,thick]
  ($(F)!0.5!(A)$) -- ++({-1/sqrt(2)},{-1/sqrt(2)})
  node[left] {$v_5=\frac1{\sqrt2}(-1,-1)$};

\draw[->,brown,thick]
  ($(A)!0.5!(B)$) -- ++({-2/sqrt(5)},{1/sqrt(5)})
  node[left] {$v_6=\frac1{\sqrt5}(-2,1)$};

\node[align=center] at (0,-4.0)
{Pyramids $C_F$
with outward unit normals $v_i$'s};

\end{tikzpicture}
\hfill
\begin{tikzpicture}[scale=0.75,>=Latex]

\foreach \x in {-4,...,4}
  \foreach \y in {-3,...,3}
    \fill[gray!60] (\x,\y) circle (1.3pt);

\draw[->] (-4.3,0) -- (4.4,0) node[right] {$x$};
\draw[->] (0,-3.3) -- (0,3.4) node[above] {$y$};

\coordinate (A) at (-3,0);
\coordinate (B) at (-2,2);
\coordinate (C) at (1,2);
\coordinate (D) at (3,0);
\coordinate (E) at (2,-2);
\coordinate (F) at (-1,-2);
\coordinate (O) at (0,0);

\fill[orange!12] (A)--(B)--(C)--(D)--(E)--(F)--cycle;
\draw[thick] (A)--(B)--(C)--(D)--(E)--(F)--cycle;

\foreach \P in {A,B,C,D,E,F}
  \fill (\P) circle (2.2pt);
\fill (O) circle (2.6pt) node[below right] {$0$};

\draw[->,red,thick]
  ($(B)!0.5!(C)$) -- ++(0,-1)
  node[pos=1.15,above right] {$u_{\rho_1}$};

\draw[->,blue,thick]
  ($(C)!0.5!(D)$) -- ++(-1,-1)
  node[pos=0.5,below right] {$u_{\rho_2}$};

\draw[->,green!60!black,thick]
  ($(D)!0.5!(E)$) -- ++(-2,1)
  node[pos=0.5,below] {$u_{\rho_3}$};

\draw[->,purple,thick]
  ($(E)!0.5!(F)$) -- ++(0,1)
  node[pos=1,left] {$u_{\rho_4}$};

\draw[->,orange,thick]
  ($(F)!0.5!(A)$) -- ++(1,1)
  node[pos=0.5,above left] {$u_{\rho_5}$};

\draw[->,brown,thick]
  ($(A)!0.5!(B)$) -- ++(2,-1)
  node[pos=0.5,above] {$u_{\rho_6}$};

\node[align=center] at (0,-4.0)
{Primitive ray generators $u_{\rho_i}$'s as inward normals};

\end{tikzpicture}

\caption{A lattice polytope with $0$ in its interior.}
\label{fig:polytopes}
\end{figure}
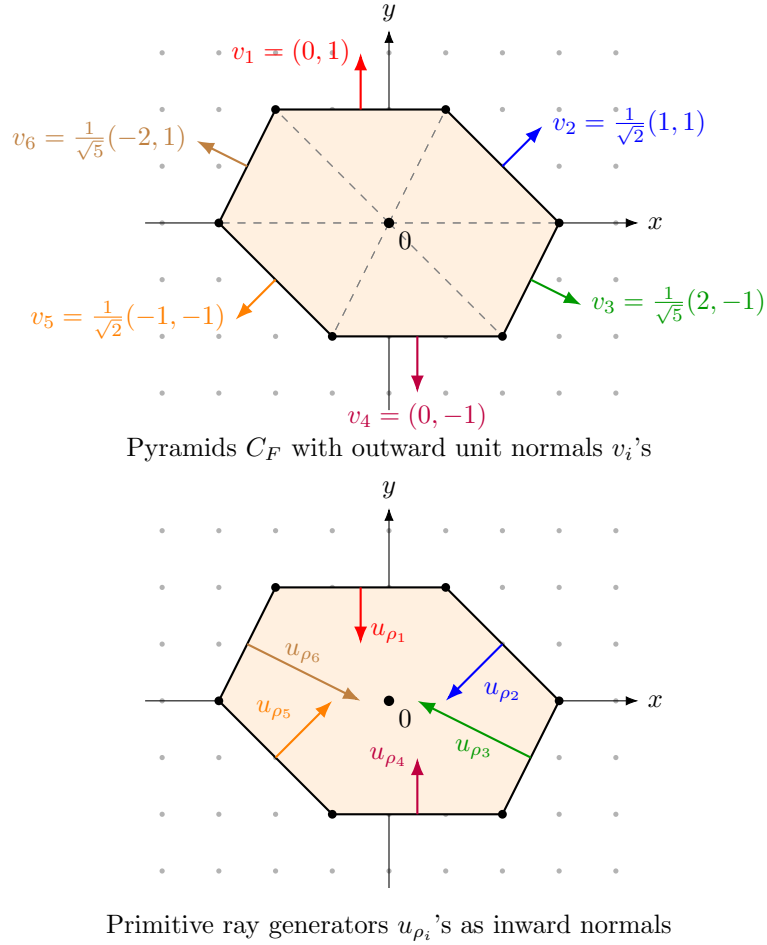

\subsection{Regular subdivision}
In Figure \ref{fig:subdivision}, we start from a two dimensional singular cone $\tau$, and then produce the subdivision into smooth cones induced by the lattice points in 
\[
N\cap\lbrace  \lambda u_\rho +\lambda' u_{\rho'}\:\vert\: 0\leq \lambda<1,\: 0\leq \lambda' <1 \rbrace.
\]
The invariant $\cA_\tau$ from \eqref{eq:Alpha tau} can then be computed from this decomposition.

\begin{figure}[H]
\centering
\begin{tikzpicture}[scale=1,>=latex]

\tikzset{
  ray/.style={->, very thick},
  lattice/.style={circle, fill=black!65, inner sep=1.8pt},
  primitive/.style={circle, draw=black, fill=black!85, line width=0.4pt,
                    minimum size=5pt, inner sep=0pt},
  gridline/.style={gray!25, very thin},
  lab/.style={font=\small}
}
\begin{scope}
\node at (2,4.9) {$\tau=\R_+\cdot(1,0)+\R_+\cdot(1,4)$};

\foreach \x in {0,...,4}
  \draw[gridline] (\x,0)--(\x,4.4);
\foreach \y in {0,...,4}
  \draw[gridline] (0,\y)--(4.4,\y);

\fill[blue!10] (0,0)--(4.2,0)--(1.1,4.4)--cycle;

\foreach \x in {0,...,4}
  \foreach \y in {0,...,4}
    \fill[lattice] (\x,\y) circle;

\draw[ray] (0,0)--(4.35,0);
\draw[ray] (0,0)--(1.1,4.4);

\node[primitive] at (1,0) {};
\node[primitive] at (1,4) {};

\node at (2,2) {$\tau$};
\node[below] at (1,0) {$u_{\rho}=(1,0)$};
\node[left] at (1,4) {$u_{\rho'}=(1,4)$};

\end{scope}

\begin{scope}[xshift=7cm]

\node at (2,4.9) {Regular subdivision};

\foreach \x in {0,...,4}
  \draw[gridline] (\x,0)--(\x,4.4);
\foreach \y in {0,...,4}
  \draw[gridline] (0,\y)--(4.4,\y);

\fill[blue!8]   (0,0)--(4.2,0)--(4.2,4.2)--cycle;
\fill[blue!14]  (0,0)--(4.2,4.2)--(2.1,4.2)--cycle;
\fill[blue!8]   (0,0)--(2.1,4.2)--(1.4,4.2)--cycle;
\fill[blue!14]  (0,0)--(1.4,4.2)--(1.1,4.4)--cycle;

\foreach \x in {0,...,4}
  \foreach \y in {0,...,4}
    \fill[lattice] (\x,\y) circle;

\draw[ray] (0,0)--(4.35,0);
\draw[ray] (0,0)--(4.2,4.2);
\draw[ray] (0,0)--(2.1,4.2);
\draw[ray] (0,0)--(1.4,4.2);
\draw[ray] (0,0)--(1.1,4.4);

\foreach \k in {0,...,4}
    \node[primitive] at (1,\k) {};

\node[below] at (1,0) {$u_{\rho}=u_{\rho_0}=(1,0)$};
\node[right] at (1,1) {$u_{\rho_1}=(1,1)$};
\node[right] at (1,2) {$u_{\rho_2}=(1,2)$};
\node[left] at (1,3) {$u_{\rho_3}=(1,3)$};
\node[left] at (1,4) {$u_{\rho'}=u_{\rho_4}=(1,4)$};

\node at (3.5,2) {$\tau_1'$};
\node at (2.5,3) {$\tau_2'$};
\node at (1.7,4) {$\tau_3'$};
\node at (1.3,4.4) {$\tau_4'$};

\end{scope}

\end{tikzpicture}

\caption{A regular subdivision of a cone $\tau$ with $\mult(\tau)=4$.}

\label{fig:subdivision}
\end{figure}
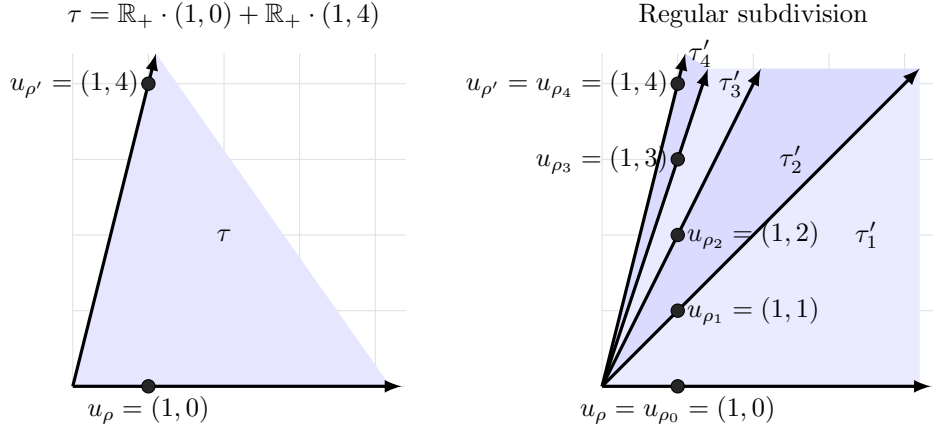

\bibliography{Ref}
\bibliographystyle{amsplain}

\end{document}